\documentclass[12pt]{amsart}

\usepackage[margin=1.15in]{geometry}

\usepackage{amsmath,amscd,amssymb,amsfonts,latexsym,wasysym, mathrsfs, mathtools,hhline, bm, tikz-cd}
\usepackage[all, cmtip]{xy}
\usepackage[utf8]{inputenc}

\usepackage{url}

\definecolor{hot}{RGB}{65,105,225}

\usepackage[pagebackref=true,colorlinks=true, linkcolor=hot ,  citecolor=hot, urlcolor=hot]{hyperref}%make all the references and links clickable

\usepackage{ textcomp }
\usepackage{ tipa }

\usepackage{graphicx,enumerate}

\usepackage{mathdots}

\newcommand{\C}{\mathbb{C}}

\theoremstyle{plain}
\newtheorem{theorem}{Theorem}[section]
\newtheorem{proposition}[theorem]{Proposition}

\newtheorem{lm}[theorem]{Lemma}

\newtheorem{corollary}[theorem]{Corollary}

\newtheorem{lemma}[theorem]{Lemma}
\newtheorem{thrm}[theorem]{Theorem}

\theoremstyle{definition}

\newtheorem{remark}[theorem]{Remark}

\newtheorem{ex}[theorem]{Example}
\newtheorem*{ex*}{Example}

\def\be{\begin{equation}}
\def\ee{\end{equation}}

\def\bt{\begin{thrm}}
\def\et{\end{thrm}}

\def\bc{\begin{cor}}
\def\ec{\end{cor}}

\def\br{\begin{rmk}}
\def\er{\end{rmk}}

\def\bp{\begin{prop}}
\def\ep{\end{prop}}

\def\bl{\begin{lm}}
\def\el{\end{lm}}

\def\bex{\begin{ex}}
\def\eex{\end{ex}}

\def\bd{\begin{defn}}
\def\ed{\end{defn}}

\newcommand\sH{{\mathcal H}}

\newcommand\sF{{\mathcal F}}

\newcommand\sL{\mathcal{L}}

\newcommand\pp{{\mathbb{P}}}

\DeclareMathOperator{\reg}{reg}                  % reg
                  
\DeclareMathOperator{\id}{id}                    % id

\DeclareMathOperator{\MLdeg}{MLdeg}

\def\tt{\mathsf{t}}
\def\pp{\mathsf{p}}
\def\uu{\mathsf{u}}

\def\bC{\mathbb{C}}

\def\bP{\mathbb{P}}
\def\cH{\mathcal{H}}
\def\bL{\mathbf{L}}

\def\cU{\mathcal{U}}

\def\cO{\mathcal{O}}
\def\lra{\longrightarrow}
\def\bQ{\mathbb{Q}}

\def\bZ{\mathbb{Z}}

\def\bs{\mathbf{s}}
\def\bE{\mathbf{E}}
\def\bF{\mathbf{F}}

\usepackage{array}
\newcolumntype{H}{>{\setbox0=\hbox\bgroup}c<{\egroup}@{}} %% Hide column of table
\usepackage[normalem]{ulem} % \sout

\title[Huh-Sturmfels Conjecture]{Logarithmic cotangent bundles, Chern-Mather classes, and the  Huh-Sturmfels Involution conjecture} 

\author{Laurentiu G. Maxim}
\address{Department of Mathematics,         University of Wisconsin-Madison,  480 Lincoln Drive, Madison WI 53706-1388, USA.}
\email {maxim@math.wisc.edu}\urladdr{https://www.math.wisc.edu/~maxim/}
\author{Jose Israel Rodriguez}
\address{Department of Mathematics,         University of Wisconsin-Madison,  480 Lincoln Drive, Madison WI 53706-1388, USA.}
\email {jose@math.wisc.edu}\urladdr{http://www.math.wisc.edu/~jose/}
\author{Botong Wang}
\address{Department of Mathematics,         University of Wisconsin-Madison,  480 Lincoln Drive, Madison WI 53706-1388, USA.}
\email {wang@math.wisc.edu}\urladdr{http://www.math.wisc.edu/~wang/}

\author{Lei Wu}
\address{Department of Mathematics, KU Leuven, 
Celestijnenlaan 200B
B-3001 Leuven, Belgium}
\email {lei.wu@kuleuven.be}\urladdr{https://sites.google.com/view/leiwuswebsite/}



\begin{document}

\date{\today}

\maketitle

\begin{abstract}  
Using compactifications in the logarithmic cotangent bundle, we obtain a formula for the Chern classes of the pushforward of Lagrangian cycles under an open embedding with normal crossing complement. This generalizes earlier results of Aluffi and Wu-Zhou. The first application of our formula is a geometric description of Chern-Mather classes of an arbitrary very affine variety, generalizing earlier results of Huh which held under the smooth and sch\"on assumptions. As the second application, we confirm 
an involution formula relating sectional maximum likelihood (ML) degrees and ML bidegrees, which was conjectured by Huh and Sturmfels in 2013. 
\end{abstract}

% \tableofcontents

%%%%%%%%%%%%%%%%%%%%%%%%%%%%%%%

\section{Introduction}\label{intro}

Maximum likelihood estimation in statistics leads to the problem of finding the critical points of a likelihood function on an algebraic variety. The number of critical points of a general likelihood function is called the \emph{maximum likelihood (ML) degree} of the algebraic variety. It is well-known that the ML degree is closely related to the topology of the variety. More precisely, if the variety is smooth, its ML degree can be computed as the Euler characteristic of an open subset of the variety, and if the variety is singular, the ML degree can be computed as the Euler characteristic of a certain constructible function which measures the complexity of singularities (see \cite{Huh}, \cite{RW1} and Corollary \ref{cor_RW}). 

Some natural generalizations of the ML degree are the ML bidegrees and the  sectional ML degrees. In \cite{HS}, Huh and Sturmfels conjectured that the ML bidegrees and the sectional ML degrees of a variety determine each other under some involution formulas, and proved the case when the variety is smooth and Sch\"on. The main goal of this paper is to confirm their Involution Conjecture in full generality. 

If a variety is smooth and Sch\"on, Huh proved that its Chern-Schwartz-MacPherson (CSM) classes determine the ML bidegrees (\cite[Theorem 2]{Huh}, see also \cite[page 101]{HS}). So the conjecture follows in this case from Aluffi's involution formula which relates the Euler characteristics of general linear sections of a projective variety to its CSM class (\cite[Theorem 1.1]{Alu2}). To generalize Huh's result to arbitrary varieties, we need to establish a similar relation between CSM classes and the ML bidegrees, which we achieve in Theorem \ref{thm_MLb}. In fact, we show that the correct substitute for the CSM classes of a smooth very affine variety are the Chern-Mather classes, i.e., the CSM class associated to the local Euler obstruction function. The key step in the proof of Theorem \ref{thm_MLb} is a new geometric formula computing the CSM classes of any constructible function on a smooth quasi-projective variety (see Theorem \ref{thm_main}). Our geometric formula involves logarithmic conic Lagrangian cycles, and reduces to earlier works of Ginsburg (\cite[Theorem 3.2]{Gin}). 

\medskip

Let $X$ be a smooth complex algebraic variety, and let $D\subset X$ be a normal crossing divisor. Denote the complement $X\setminus D$ by $U$, and let $j:U\hookrightarrow X$ be the open  inclusion. Let $\Omega_X^1(\log D)$ be the sheaf of algebraic one-forms with logarithmic poles along $D$, and denote the total space of the corresponding vector bundle by $T^*(X, D)$. Clearly, $T^*(X, D)$ contains $T^*U$ as an open subset. Given a conic Lagrangian cycle $\Lambda$ in $T^*U$, we denote its closure in $T^*(X, D)$ by $\overline\Lambda^{\log}$. 

Recall that, by using a microlocal interpretation of MacPherson's Chern class transformation (see \cite{Gin, AMSS}), one can associate Chern classes $c_i^E(\Lambda) \in A_i(X)$, $i=1,\ldots,k$, in Chow homology to any conic irreducible $k$-dimensional subvariety  of a rank $k$ vector bundle $E$ on $X$ (see Section \ref{sec:csm} for a brief description of this construction). With these notations, we prove the following.

\begin{theorem}\label{thm_main}
Let $\sF^\centerdot$ be any constructible sheaf complex  on $U$. Then 
\[
c^{T^*(X, D)}_*\Big(\overline{CC(\sF^\centerdot)}^{\log}\Big)=c^{T^*X}_*\big(CC(Rj_*\sF^\centerdot)\big) \in A_*(X),
\]
where $CC(-)$ denotes the characteristic cycle of $-$, and if $CC(\sF^\centerdot)=\sum_{k}n_k\Lambda_k$, then $\overline{CC(\sF^\centerdot)}^{\log}\coloneqq\sum_{k}n_k\overline{\Lambda_k}^{\log}$.
\end{theorem}

\begin{remark}
When the conic Lagrangian cycle $\Lambda$ is equal to the zero section of $T^*U$, this result is well known, e.g., see \cite{Al1, Al2}. For an arbitrary conic Lagrangian cycle, the equality of the top Chern class in Theorem \ref{thm_main} was proved by Zhou and the last author in \cite{WZ}. 
\end{remark}

Recall that an affine variety $Z$ is called {\it very affine}, if it admits a closed embedding to an affine torus $(\C^*)^n$ for some $n$. In this paper, for a very affine variety we always assume that such a closed embedding is chosen. A \emph{master function}\footnote{Such master function is also called a likelihood function in \cite{Huh}. In this paper, we reserve the notion of likelihood function for the ones also containing the factor $(1-x_1-\cdots -x_n)^{u_0}$, which is consistent with the conventions in \cite{HS}.} on $(\C^*)^n$ is of the form 
\[
m_{\mathbf{w}}\coloneqq x_1^{w_1}\cdots x_n^{w_n},
\]
where $(x_1,\ldots,x_n)$ are the coordinate functions on $(\C^*)^n$ and  $\mathbf{w}=(w_1, \ldots, w_n)\in \bZ^n$. If, more generally, $\mathbf{w}\in \C^n$, then $m_{\mathbf{w}}$ is a multivalued function. Nevertheless, the critical points of $m_{\mathbf{w}}|_{Z_{\textrm{reg}}}$ are well defined, where $Z_{\textrm{reg}}$ denotes the smooth locus of $Z$. In fact, the critical points do not depend on the choice of local branches of the function, and they are equal to the degeneration points of the restriction of the holomorphic 1-form 
\[
d\log m_{\mathbf{w}}=w_1\frac{dx_1}{x_1}+\cdots +w_n\frac{dx_n}{x_n}
\]
to $Z_{\textrm{reg}}$. The total space of all critical points of the master functions defines a closed subvariety of $Z_{\textrm{reg}}\times \bC^{n}$:
\[
\mathfrak{X}^\circ(Z)=\{(z, \mathbf{w}) \in Z_{\textrm{reg}}\times \C^n\mid z \text{ is a critical point of } m_{\mathbf{w}}|_{Z_{\textrm{reg}}}\}.
\]
Using the natural compactifications $(\C^*)^n\subset \bP^n$ and $\C^n\subset \bP^n$, we can consider $Z_{\textrm{reg}}\times \C^n$ as a locally closed subvariety of $\bP^n\times \bP^n$. Let $\mathfrak{X}(Z)$ be the closure of $\mathfrak{X}^\circ(Z)$ in $\bP^n\times \bP^n$. 

As the first application of Theorem \ref{thm_main}, we prove a geometric formula relating the Chern-Mather classes of $Z$ and the bidegrees of $\mathfrak{X}(Z)$, generalizing \cite[Theorem 1.2]{Huh}.

\begin{theorem}\label{thm_CSM}
Given a very affine variety $Z\subset (\bC^*)^n$ of dimension $d$, let $\mathfrak{X}(Z)$ be defined as above. Then
\[
[\mathfrak{X}(Z)]=\sum_{i=0}^{d}v_i[\bP^{i}\times \bP^{n-i}]\in A_*(\bP^n\times \bP^{n}),
\]
where 
\[
c_{Ma}(Z)=\sum_{i=0}^d(-1)^{d-i} v_i [\bP^{i}]\in A_*(\bP^n).
\]
Here, the Chern-Mather class $c_{Ma}(Z)$ is defined as $c_*(Eu_Z)$, where $c_*$ is the MacPherson Chern class transformation and $Eu_Z$ is the local Euler obstruction function of $Z$, regarded as a constructible function on $\bP^n$.
\end{theorem}

\begin{remark}
In the original statement of \cite[Theorem 2]{Huh}, the total space of critical points $\mathfrak{X}^\circ(Z)$ is defined as a subvariety of $Z\times \bP^{n-1}$, and hence $\mathfrak{X}(Z)$ is a subvariety of $\bP^n\times \bP^{n-1}$. When $Z$ is not equal to the ambient space $(\bC^*)^n$, our definition of $\mathfrak{X}(Z)$ is a cone of the one in \cite{Huh}. Hence, in this case, the two constructions define the same sequence of numbers $v_i$. We choose the new construction because it gives the correct formula even when $Z$ is equal to the ambient space $(\bC^*)^n$, as well as due to our use of CSM classes of conic cycles  (see \eqref{mce}). In fact, to understand the Chern classes of conic cycles $\Lambda$ in a vector bundle $E$, one will lose track of all conic cycles supported on the zero section of the vector bundle if taking the projective cones $\bP(\Lambda)\subset \bP(E)$ instead of taking the closure $\overline{\Lambda}\subset \overline{E}=\bP(E\oplus \C)$. For the same reason, in the later definition of the likelihood correspondence variety $\sL_Y$, we take the closure and define it as a subvariety of $\bP^n\times \bP^{n+1}$ instead of $\bP^n\times \bP^{n}$ as in \cite{HS}.
\end{remark}

Before stating the remaining results, we recall the definitions of the ML degrees, ML bidegrees and sectional ML degrees. 

Let $Y$ be an irreducible subvariety of $(\C^*)^n$, and let $Y^\circ=Y\setminus \cH$, where $\cH=\{x_1+\cdots +x_n=1\}$ and $(x_i)_{i=1,\ldots,n}$ are the coordinate functions of $(\C^*)^n$.\footnote{A more natural way to think of $Y^\circ$ is to realize it as a closed subvariety of $(\C^*)^{n+1}$ with the additional coordinate $x_0=1-x_1-\cdots -x_n$. Comparing to \cite[Page 69]{HS}, our coordinate function $x_i$ is equal to $\frac{p_i}{p_+}$ for $i=0,\dots,n$.} For a data point $\mathbf{u}=(u_0,u_1, \ldots, u_n)\in \C^{n+1}$, the associated (multivalued) \emph{likelihood function} is defined as 
\[
\ell_{\mathbf{u}}\coloneqq x_1^{u_1}\cdots x_n^{u_n} (1-x_1-\cdots-x_n)^{u_0}.
\]
 Assume that $Y^\circ$ is nonempty. Then the \emph{ML degree} of $Y$, denoted by $\MLdeg(Y)$, is defined as the number of critical points of $\ell_{\mathbf{u}}|_{Y^\circ_{\textrm{reg}}}$ for a generic data point $\mathbf{u}$. 

The \emph{likelihood correspondence variety} $\sL_Y$ is defined analogously to $\mathfrak{X}(Y)$. Firstly, we define
\[
\sL_Y^\circ\coloneqq \{(y, \mathbf{u})\in Y^\circ_{\textrm{reg}}\times \C^{n+1}\mid y \text{ is a critical point of }\ell_{\mathbf{u}}|_{Y^\circ_{\textrm{reg}}}\}. 
\]
Then we let $\sL_Y$ be the closure of $\sL^\circ_Y$ in $\bP^n\times \bP^{n+1}$. The \emph{ML bidegrees} of $Y$ are defined as the bidegrees of $\sL_Y\subset \bP^n\times \bP^{n+1}$. More precisely, the $i$-th ML bidegree $b_i$ of $Y$ is determined by
\[
[\sL_Y]=\sum_{i=0}^{\dim Y} b_i[\bP^i\times \bP^{n+1-i}]\in A_*(\bP^n\times \bP^{n+1}).
\]
In particular, the 0-th bidegree is equal to the ML degree, that is, $b_0=\MLdeg(Y)$. 

The \emph{sectional ML degrees} are defined to be the ML degrees of iterated hyperplane sections of $Y$. More precisely, the $i$-th sectional ML degree of $Y$ is defined by
\[
s_i\coloneqq \mathrm{MLdeg}(Y\cap L_{n-i}).
\]
where $L_{n-i}$ is a general affine subspace of $\bC^n$ of codimension $i$.
In particular, $s_0$ is the ML degree of $X$ and $s_{\dim(X)}$ is the degree of $X$. 

In the next theorem, we confirm the involution formulas conjectured by Huh and Sturmfels \cite[Conjecture 3.15]{HS} relating the ML bidegrees and the sectional ML degrees. 
\begin{theorem}\label{thm:ML-involution}
Let $Y \subset (\bC^*)^n$ be an irreducible very affine variety of dimension $d$, which is not contained in the hyperplane $\sH\coloneqq \{x_1+\cdots +x_n=1\}$. Then
\[
B_Y(\pp, \uu)=\frac{\uu\cdot S_Y(\pp, \uu-\pp)-\pp\cdot S_Y(\pp, 0)}{\uu-\pp}
, \quad
S_Y(\pp, \uu)=\frac{\uu\cdot B_Y(\pp, \uu+\pp)+\pp\cdot B_Y(\pp, 0)}{\uu+\pp},
\]
where 
\[S_Y(\pp, \uu)=(s_0 \cdot \pp^{d}+s_1\cdot \pp^{d-1}\uu+\cdots +s_d\cdot  \uu^d)\cdot \pp^{n-d}\]
and 
\[
B_Y(\pp, \uu)=(b_0 \cdot \pp^{d}+b_1\cdot \pp^{d-1}\uu+\cdots +b_d\cdot  \uu^d)\cdot \pp^{n-d}.
\]
\end{theorem}
When $Y^{\circ}=Y\setminus \sH$ is smooth and Sch\"on, the above theorem is proved by Huh (\cite{Huh}, see also \cite[page 101]{HS}). 

By using Aluffi's involution formula \cite{Alu2}, we reduce the above Theorem \ref{thm:ML-involution} to the following result, relating Chern-Mather classes and ML bidegrees. 

\begin{theorem}\label{thm_MLb}
Let $Y$ and $\cH$ be defined as in Theorem \ref{thm:ML-involution}, and let $Y^\circ=Y\setminus \cH$.
Suppose the total Chern-Mather class of $Y^\circ$ is given by 
\[
c_{Ma}(Y^\circ)=\sum_{i=0}^{\dim Y} (-1)^{\dim Y-i}\alpha_i [\bP^{i}] \in A_*(\bP^n).
\]
\noindent Then, the sequence $\alpha_0, \ldots, \alpha_{\dim Y}$ consists of the ML bidegrees of $Y$. In other words, 
\[
B_Y(\pp, \uu)=\sum_{i=0}^{\dim Y} \alpha_i\,\pp^{n-i}\uu^i.
\]
\end{theorem}

Since the degree zero part of the Chern-Schwartz-MacPherson class (viewed in $A_*(\bP^n)$) computes the Euler characteristic, and since $\alpha_0$ is equal to the ML degree, an immediate consequence of Theorem \ref{thm_MLb} is the following corollary, which is proved implicitly in \cite{RW1}. 

\begin{corollary}\label{cor_RW}
Let $Y$ be an irreducible subvariety of $(\C^*)^n$. Assume that $Y^\circ=Y\setminus \sH$ is nonempty. Then
\[
\MLdeg(Y)=(-1)^{\dim Y}\chi(Eu_{Y^\circ}). 
\]
\end{corollary}

The paper is organized as follows. 
Section \ref{sec:2} is devoted to presenting background material on constructible functions, constructible sheaves, and their characteristic cycles, on the microlocal interpretation of the MacPherson's Chern classes, as well as on Aluffi's involution formula relating CSM classes to sectional Euler characteristics. Theorem \ref{thm_main} is proved in Section \ref{sec:3}. Applications of Theorem \ref{thm_main} are discussed in Section \ref{sec:4}, where Theorems \ref{thm_CSM}, \ref{thm:ML-involution} and \ref{thm_MLb} are proved and examples using numerical algebraic geometry are worked out. 

 \medskip

Throughout the paper, we work in the category of algebraic varieties defined over $\bC$. In particular, a vector bundle corresponds to a locally free coherent sheaf. 

\medskip

\textbf{Acknowledgements.} The authors thank J\"org Sch\"urmann for his comments on an earlier version of the manuscript. Rodriguez thanks Daniel Corey for his comments on the Sch\"on property. Maxim is partially supported by the Simons Foundation (Collaboration Grant \#567077), and by the Romanian Ministry of National Education (CNCS-UEFISCDI grant PN-III-P4-ID-PCE-2020-0029). Wang is partially supported by a Sloan fellowship. Wu is supported by an FWO postdoctoral fellowship. 

%%%%%%%%%%%%%%%%%%%%%%%

\section{Preliminaries}\label{sec:2}
In this section, we collect relevant background material on constructible functions, constructible sheaves and their characteristic cycles, we give a brief overview of the microlocal interpretation of MacPherson's Chern classes, and recall Aluffi's involution formula relating CSM classes to sectional Euler characteristics.

\subsection{Constructible sheaves. Constructible functions. Characteristic cycles}

Denote by $D^b_c(X)$ the bounded derived category of $\bC$-constructible complexes (with respect to some stratification) on the smooth complex algebraic variety $X$. By associating characteristic cycles to constructible complexes on $X$ (e.g., see \cite[Definition 4.3.19]{Di} or \cite[Chapter IX]{KS}), one gets a functor 
$$CC:K_0(D^b_c(X)) \lra L(X)$$
on the Grothendieck group of $\bC$-constructible complexes, where $L(X)$ is the free abelian group spanned by the irreducible conic Lagrangian cycles in the cotangent bundle $T^*X$. Recall that any element of $L(X)$ is of the form $\sum_k n_k \cdot T^*_{Z_k}X$, for some $n_k \in \bZ$ and $Z_k$ closed irreducible subvarieties of $X$. Here, if $Z$ is a closed irreducible subvariety of $X$ with smooth locus $Z_{\reg}$, its conormal bundle $T^*_{Z}X$ is defined as the closure in $T^*X$ of $$T^*_{Z_{\reg}}X:=\{
(z,\xi)\in T^*X \mid z \in Z_{\reg}, \ \xi \in T^*_zX, \ \xi\vert_{T_zZ_{\rm reg}}=0 \}.$$
One can then define a group isomorphism
$$T:L(X) \lra Z(X)$$ to the group $Z(X)$ of algebraic cycles on $X$ by:
$$\sum_k n_k \cdot T^*_{Z_k}X \longmapsto \sum_k (-1)^{\dim Z_k} n_k Z_k.$$

Let $F(X)$ be the group of algebraically constructible functions on $X$, i.e., the free abelian group generated by indicator functions $1_Z$ of closed irreducible subvarieties $Z$ of $X$. There is a unique linear map $$\chi:F(X) \longrightarrow \bZ$$
called the {\it Euler characteristic}, defined on generators by $\chi(1_Z):=\chi(Z).$

An important example of a constructible function on $X$ is the MacPherson {\it local Euler obstruction} function $Eu_Z$ of an irreducible subvariety $Z$ of $X$, see  \cite{MP0}. The local Euler obstruction function is a measure of the singularities of $Z$, and it takes the value $1$ on the smooth locus $Z_{\rm reg}$.

 The relation between constructible complexes and constructible functions is made explicit by the following construction. 
  To any constructible complex $\sF^{\centerdot} \in D^b_c(X)$,  one associates a constructible function $\chi_{st}(\sF^{\centerdot})\in F(X)$ by taking stalkwise Euler characteristics, i.e.,
$$\chi_{st}(\sF^{\centerdot})(x):=\chi(\sF^{\centerdot}_x)$$
for any $x \in X$. For example, $\chi_{st}(i_!\bC_Z)=1_Z$, for $Z$ a closed irreducible subvariety of $X$. Note that if $\varphi=\chi_{st}(\sF^{\centerdot})$, then $\chi(\varphi)=\chi(X,\sF^{\centerdot})$. 

Since the Euler characteristic is additive with respect to distinguished triangles, one gets an induced group homomorphism (in fact, an epimorphism)
$$\chi_{st}:K_0(D^b_c(X)) \lra F(X).$$
Moreover, since the class map $D^b_c(X) \to K_0(D^b_c(X))$ is onto, $\chi_{st}$ is already an epimorphism on $D^b_c(X)$.

If $Z$ is a closed subvariety of $X$, we may regard the function $Eu_Z$ as being defined on  all of $X$ by setting $Eu_Z(x)=0$ for $x \in X \setminus Z$. In particular, one may consider the group homomorphism
\begin{equation}\label{eq_ZCF}
Eu:Z(X) \lra F(X)
\end{equation}
defined on an irreducible cycle $Z$ by the assignment $Z \mapsto Eu_Z$, and then extended by $\bZ$-linearity. A well-known result (e.g., see \cite[Theorem 4.1.38]{Di} and the references therein) states that the homomorphism
$Eu:Z(X) \to F(X)$
is an isomorphism.

The Euler obstruction function appears in the formulation of the {\it local index theorem}, which in the  notations of this section asserts the existence of the following commutative diagram  (e.g., see \cite[Section 5.0.3]{S} and the references therein):
\begin{equation}\label{eq_xy}
\xymatrix{
K_0(D^b_c(X)) \ar[d]_{CC} \ar[r]^{\chi_{st}} & F(X) \ar[d]^{Eu^{-1}}_{\cong}  \\
L(X) \ar[r]^T_{\cong} & Z(X)
}
\end{equation}
In particular, one can associate a characteristic cycle to any constructible function $\varphi \in F(X)$ by the formula
$$CC(\varphi):=T^{-1} \circ Eu^{-1}(\varphi).$$
For example, if $Z$ is a closed irreducible subvariety of $X$, one has: \be\label{cc} CC(Eu_Z)=(-1)^{\dim_{\C} Z}\cdot T^*_Z X.\ee
Note also that $$CC(\sF^{\centerdot})=CC(\chi_{st}(\sF^{\centerdot}))$$ for any constructible complex $\sF^{\centerdot} \in D^b_c(X)$.

It is well known (e.g., see \cite[Section 2.3]{S}) that all the usual functors in sheaf theory, which respect the corresponding category of constructible complexes of sheaves, induce by the epimorphism $\chi_{st}$ well-defined group homomorphisms on the level of constructible functions. Moreover, if $f:X \to Y$ is a morphism of complex algebraic varieties, then one has the equality (see, e.g., \cite[formula (112)]{MS} and the references therein)
$$Rf_!=Rf_*:K_0(D^b_c(X)) \lra K_0(D^b_c(Y)),$$
which also implies for the induced group homomorphisms of complex algebraically constructible functions the equality
\begin{equation}\label{equ}
   f_!=f_*:F(X) \lra F(Y). 
\end{equation}
Here we note that $f_*:F(X) \lra F(Y)$ can be described more explicitly as: 
$$1_Z \longmapsto \left( y \longmapsto \chi(f^{-1}(y)\cap Z), \ \ y \in Y\right).$$
Finally, by \eqref{eq_xy}, all these functors can also be considered as functors on conic Lagrangian cycles in the cotangent bundle $T^*X$ (with support in a certain subvariety, if needed).

\subsection{CSM classes. Microlocal interpretation}\label{sec:csm}
We work in the complex algebraic context, with $A_*$ the Chow group and $H_*$ the Borel-Moore homology.

In \cite{MP0}, MacPherson extended the definition of Chern classes to singular complex algebraic varieties. More precisely, he defined a natural transformation
$$c_*:F(-) \lra A_*(-)$$
from the functor of constructible functions (with proper morphisms) to Chow (or Borel-Moore) homology, such that if $X$ is a smooth variety then $c_*(1_X)=c(TX) \cap [X]$. Here, $c(TX)$ denotes the total (cohomology) Chern class of the tangent bundle $TX$, and $[X]$ is the fundamental class.
For any %constructible subset 
locally closed irreducible subvariety
$Z$ of a complex algebraic variety $X$, the class 
$$c_{SM}(Z):=c_*(1_Z) \in A_*(X)$$
is usually referred to as the {\it Chern-Schwartz-MacPherson (CSM) class} of $Z$ in $X$. Similarly, 
the class $$c_{Ma}(Z):=c_*(Eu_Z) \in A_*(X)$$ is called the {\it Chern-Mather class} of $Z$, where we regard $Eu_Z$ as a constructible function on $X$ by setting the value zero on $X \setminus Z$.

Results of Ginsburg \cite{Gin} and Sabbah \cite{Sab} showed that McPherson's Chern class transformation $c_*$ factors through the group of conic Lagrangian cycles in the cotangent bundle. This construction was  revisited more recently in \cite{Alu}, as well as in \cite{AMSS}, also in the equivariant context. 

\medskip

In this section, we recall the  construction of MacPherson's Chern class transformation in terms of characteristic cycles, following the approach of \cite{Alu, AMSS}.

Let $X$ be a smooth complex algebraic variety, and let $E$ be a rank $r$ vector bundle on $X$. Let $\overline{E}\coloneqq \bP(E\oplus \mathbf{1})$ be the projective bundle, which is a fiber-wise compactification of $E$ (here $\mathbf{1}$ denotes the trivial line bundle on $X$). Then $E$ may be identified with the open complement of $\bP(E)$ in $\overline{E}$. Let $\pi:E\to X$ and $\bar{\pi}:\overline{E} \to X$ be 
the projections, and let $\xi:=c_1(\mathcal{O}_{\overline{E}}(1))$. Pullback via ${\bar \pi}$ realizes $A_*(\overline{E})$ as a $A_*(X)$-module. Moreover, as shown in \cite[Theorem 3.3]{Ful}, each $\alpha \in A_i(\overline{E})$ can be uniquely written as:
\begin{equation}\label{sh}
    \alpha=\sum_{j=0}^r \xi^j \cap {\bar \pi}^* \alpha_j,
\end{equation}
for $\alpha_j \in A_{i-r+j}(X)$.

An irreducible conic $r$-dimensional subvariety ($r={\rm rank} \ E$) $\Lambda \subset E$
determines an $r$-dimensional cycle $\overline{\Lambda}$ in $\overline{E}$ and, by formula \eqref{sh}, knowledge of $[\overline{\Lambda}] \in A_r(\overline{E})$ is equivalent to knowledge of a collection of $r+1$ classes on $X$. We denote these classes by $$c_0^E(\Lambda), \ldots, c_r^E(\Lambda)$$ with $c_j^E(\Lambda) \in A_j(X)$, and call them the {\it Chern classes of $\Lambda$}. The terminology is justified by the following result (see \cite[Lemma 4.3]{Alu}, \cite[Proposition 3.3]{AMSS} and the references therein), applied to the cotangent bundle $T^*X$ and elements of $L(X)$:
\begin{proposition}\label{pmc}
For any constructible function $\varphi \in F(X)$, the Chern classes of the characteristic cycle $CC(\varphi)$ equal the signed MacPherson Chern classes of $\varphi$, namely:
\begin{equation}\label{mc}
    c_j^{T^*X}\left(CC(\varphi)\right)=(-1)^j \cdot  c_j(\varphi) \in A_j(X), \ \ j=0,\ldots, \dim(X),
\end{equation}
where $c_j(\varphi)$ denotes the $j$-th component of MacPherson's  Chern class $c_*(\varphi)$.
\end{proposition}

\begin{remark}
The signs appearing in  \eqref{mc} may of course be already absorbed in the definition of the $c_j(\Lambda)$. This is in fact the way Ginzburg describes these classes in the Appendix of \cite{Gin}.
\end{remark}

For future reference, let us introduce the following notation for the {\it signed MacPherson Chern classes} appearing in Proposition \ref{pmc}. For $\varphi \in F(X)$, set:
\begin{equation}
    \check{c}_{*}(\varphi)=\sum_{j \geq 0} \check{c}_j(\varphi):=
    \sum_{j \geq 0} (-1)^j {c}_j(\varphi).
\end{equation}
In particular, if $\varphi=1_Z$ or $\varphi=Eu_Z$, we get a corresponding {\it signed CSM class} $\check{c}_{SM}(Z)$ and, respectively, a {\it signed Chern-Mather class} $\check{c}_{Ma}(Z)$ of a 
locally closed irreducible subvariety $Z \subset X$.

\medskip

There is an alternative way of recovering the Chern classes of lagrangian cycles, which makes use of $\bC^*$-equivariant Chow groups.

Consider the $\bC^*$-action on the vector bundle $E$ by fiberwise dilation, and the trivial action on $\mathbf{1}$. This induces a $\bC^*$-action on $\overline{E}$, such that the inclusion $E \subset \overline{E}$ is $\bC^*$-equivariant (and the trivial action on $\bP(E)$). The natural projection $\pi:E\to X$ is equivariant, where $\bC^*$ acts trivially on $X$. 
Let $A_*^{\bC^*}(E)$ and $A_*^{\bC^*}(X)$ be the $\bC^*$-equivariant Chow groups of $E$ and $X$, respectively. Since $\bC^*$ acts trivially on $X$, we have an isomorphism
$$A_*^{\bC^*}(X) \cong A_*(X)[t],$$ where $t:=c_1(\mathcal{O}_{\bP^\infty}(-1))$. Moreover, the inclusion $i:X \hookrightarrow E$ of the zero section induces isomorphisms (e.g., see \cite[Lemma 2.5]{AMSS}):
$$i^*=(\pi^*)^{-1}:A_*^{\bC^*}(E) \lra A_{*-r}^{\bC^*}(X),$$
with $r={\rm rank} \ E$.

An $r$-dimensional $\bC^*$-invariant cycle $\Lambda$ in $E$ determines as above an $r$-dimensional cycle $\overline{\Lambda}$ in $\overline{E}$. By comparing the class $[\Lambda]$ of $\Lambda$  in the equivariant Chow group  $A_*^{\bC^*}(E)\overset{i^*}{\cong}A_{*-r}^{\bC^*}(X)$ with the class $[\overline{\Lambda}]$ in the ordinary Chow group $A_*(\overline{E})$, one gets the following identification of \cite[Proposition 2.7]{AMSS}:
\begin{equation}\label{mce}
    i^*([\Lambda])\vert_{t\mapsto 1}=c_0^E(\Lambda) + c_1^E(\Lambda)\ldots + c_r^E(\Lambda).
\end{equation}
Formula \eqref{mce} yields immediately the following.
\begin{lemma}\label{lemma_functoriality}
Let $\phi:E_1 \to E_2$ be a bundle map of rank $r$ vector bundles over $X$. If $\Lambda$ is an $r$-dimensional $\bC^*$-invariant cycle in $E_2$ such that the $\bC^*$-invariant subset $\phi^{-1}(\Lambda)$ in $E_1$ is also pure $r$-dimensional, then
$$c^{E_1}_j(\phi^*\Lambda)=c^{E_2}_j(\Lambda),$$
for each $j=0,\ldots,r$.
\end{lemma}

\subsection{Aluffi's inversion formula}\label{inv}
Let $X$ be a locally closed set in $\bP^n$, so that the function $1_X$ is constructible on $\bP^n$. Then $$c_{SM}(X):=c_*(1_X)=\sum_{j\geq 0} c_j [\bP^j] \in A_*(\bP^n).$$ Let
$$\gamma_X(\tt):=\sum_{j\geq 0} c_j \tt^j$$
be the polynomial of degree $\leq n$ obtained from $c_{SM}(X)$ by replacing $\bP^j$ by $\tt^j$. Let $X_j=X \cap L_{n-j}$, where $L_{n-j}$ is a generic linear subspace of codimension $j$ in $\bP^n$. Consider the generating polynomial of degree $\leq n$ of the Euler characteristics of these sections, defined as
$$\chi_X(\tt):=\sum_{j \geq 0} \chi(X_j) \cdot (-\tt)^j.$$

In \cite{Alu2}, Aluffi showed that for any locally closed set $X$ in $\bP^n$, the polynomials $\gamma_X(\tt)$ and $\chi_X(\tt)$ carry precisely the same information. 
In order to formulate the result from loc. cit., consider the following linear transformation:
$$p(\tt) \longmapsto \mathcal{I}(p):=\frac{\tt \cdot p(-\tt-1)+p(0)}{\tt+1},$$
and note that if $p(\tt)$ is a polynomial, then $\mathcal{I}(p)$ is a polynomial of the same degree. Furthermore, $\mathcal{I}$ is an {\it involution}, whose effect is to perform a sign-reversing symmetry about $\tt=-1/2$ of the non-constant part of $p$. The main result of \cite{Alu2} is the following:
\begin{theorem}\label{thm_Aluffi}
For every locally closed set $X$ in $\bP^n$, the involution $\mathcal{I}$ interchanges $\gamma_X(\tt)$ and $\chi_X(\tt)$, i.e.,
\begin{equation*}
    \gamma_X=\mathcal{I}(\chi_X), \ \ 
    \chi_X=\mathcal{I}(\gamma_X).
\end{equation*}
\end{theorem}

More generally, consider a constructible function $\varphi$ on $\bP^n$. Assuming that
\[
c_*(\varphi)=\sum_{j\geq 0} c_j [\bP^j] \in A_*(\bP^n),
\]
we define 
\[
\gamma_\varphi(\tt)\coloneqq \sum_{j\geq 0} c_j \tt^j.
\]
Let $\varphi_j$ be the restriction of $\varphi$ to a generic codimension $j$ linear subspace $L_{n-j}$, and we  consider $\varphi_j$ as a constructible function on $\bP^n$ with support contained in $L_{n-j}$. Then we define
\[
\chi_\varphi(\tt)\coloneqq \sum_{j \geq 0} \chi(\varphi_j) \cdot (-\tt)^j.
\]
Since the constructible functions of the form $1_X$ for irreducible closed subvarieties $X\subset \bP^n$ form a basis of $F(\bP^n)$, i.e., the abelian group of all constructible functions on $\bP^n$, Theorem \ref{thm_Aluffi} can be reformulated as the following corollary. 
\begin{corollary}\label{cor_Aluffi}
For any constructible function $\varphi$ on $\bP^n$, the involution $\mathcal{I}$ interchanges $\gamma_\varphi(\tt)$ and $\chi_\varphi(\tt)$, i.e.,
\begin{equation*}
    \gamma_\varphi=\mathcal{I}(\chi_\varphi), \ \ 
    \chi_\varphi=\mathcal{I}(\gamma_\varphi).
\end{equation*}
\end{corollary}

%%%%%%%%%%%%%%%%%%%%%%

\section{Characteristic cycles and the universal graph embedding}\label{sec:3}

In this section we prove Theorem \ref{thm_main}.

In \cite[Section 3]{Gin}, Ginsburg gave explicit pushforward formula for conic Lagrangian cycles on an open embedding of the complement of a hypersurface. The formula uses explicitly the defining equation of the hypersurface. First, we modify Ginsburg's construction to obtain a global formula, which does not depend on the choice of local defining equations. 

Let $X$ be a smooth complex algebraic variety and let $D=\sum_{i=1}^r D_i$ be a sum of effective divisors. Let $U=X\setminus D$.
For each $i$, denote the total space of $\cO_X(D_i)$ by $\bL_i$, and denote the bundle map by $\pi_i: \bL_i\to X$. 
Since $D_i$ is effective, there is a tautological section of $\bL_i$ defining the divisor $D_i$, which we denote by $\bs_i: X\to \bL_i$. 

Let $\bE$ be the total space of the vector bundle $\cO_X(D_1)\oplus\cdots \oplus\cO_X(D_r)$. Then,
\[
\bE=\bL_1\times_ X \bL_2 \times_X \cdots \times_X\bL_r,
\]
and we denote the bundle map by $\pi: \bE\to X$. Putting the sections $\bs_{i}$ together gives a section of $\bE$:
\[\bs\colon X\hookrightarrow \bE, \quad x\mapsto (\bs_{1}(x), \bs_{2}(x),\dots,\bs_{r}(x)).\]
Pulling back the zero section $X\hookrightarrow \bL_i$ through the (vector bundle) morphism $\bE\to \bL_i$ gives a smooth divisor on $\bE$, which we denote by $\bF_i$. Let $\bF=\sum_{i=1}^r\bF_i$. The following lemma is tautological.
\begin{lemma}\label{lm:tautllog}
Under the above notations, $\bF$ is a simple normal crossing divisor on $\bE$. Moreover, $\bs_i^* \bF_i=D_i$ and $\bs^*\bF=D$. 
\end{lemma}

Next, we consider the local picture. Suppose that on some open set $\mathcal U\subseteq X$, there exist regular functions $f_i$, $i=1,\ldots,r$, such that the zero locus of each $f_i$ is equal to $D_i$. Notice that the section $\bs_i$ of $\bL_i$ is defined as the section ``$\mathbf{1}$" of $\cO_X(D_i)$. Thus, $\frac{1}{f_i}\cdot \mathbf{1}$ defines a nonzero section of $\bL_i$ on $\mathcal{U}$, and $(\frac{1}{f_1}, \ldots, \frac{1}{f_r})$ gives a local trivialization $\bE|_\cU\simeq \cU\times \C^r$. Under this trivialization, the section $\bs$ can be written as the graph-embedding:
\[\bs\colon \cU\hookrightarrow \cU\times \C^r, \quad x\mapsto (x,f_1(x),\dots, f_r(x)).\]
Let $\underline{t}=(t_1, \ldots, t_r)$ be the coordinates of the factor $\C^r$, and we write $\C^r_{\underline{t}}$ to emphasize the chosen coordinates. 
Under the local trivialization $\bE|_\cU\simeq\cU\times \C^r$, we have the induced local trivializations
for the logarithmic cotangent bundle
\[T^*(\bE,\bF)|_{\cU\times \C^r}\simeq T^*\cU\times \C^r_{\underline{t}}\times \C^r_{\underline{s}}\]
where $\C^r_{\underline{t}}\times \C^r_{\underline{s}}$ is the total space of the logarithmic cotangent bundle of $\C^r_{\underline{t}}$ with respect to the coordinate divisors, and $\underline{s}=(s_1, \ldots, s_r)$ with each $s_i$ corresponding to the logarithmic tangent vector $-t_i\partial_{t_i}$.

Using the section $\bs: X\to \bE$, we can identify $X$ as the closed subvariety $\bs(X)$ of $\bE$, and hence open subvarieties of $X$ as a locally closed subvarieties of $\bE$. By the above trivialization, we have
\be\label{eq:trivizlogpb}
T^*(\bE,\bF)|_{\mathcal U}=\bs^*(T^*(\bE,\bF)|_{\cU\times \C^r})\simeq T^*\cU\times \C_{\underline{s}}^r.
\ee
Let $\Lambda\subseteq T^*U$ be an irreducible conic Lagrangian cycle. Following \cite[\S 2.1]{Gin}, locally on $\cU^\circ\coloneqq\cU\setminus D$, we can define an  $(n+r)$-cycle in $T^*\cU^\circ\times \C_{\underline{s}}^r$:
\begin{equation}\label{eq_sharp}
\Lambda^\sharp|_{\cU^\circ}\coloneqq\left\{\left(x,\xi+\sum_{i=1}^r s_id\log f_i(x),s\right)\middle| (x,\xi)\in \Lambda \textup{ and } s_i\not=0 \textup{ for all }i \right\}.
\end{equation}
It follows from the next lemma that the definition of $\Lambda^\sharp|_{\cU^\circ}$ does not depend on the choice of the functions $f_i$ and the local cycles glue together to a global $(n+r)$-cycle 
\[
\Lambda^\sharp\subseteq T^*U\times\C^r\simeq T^*(\bE,\bF)|_U.
\]
Moreover, it also follows from the next lemma that $\Lambda^\sharp$ is a conic cycle. 

\begin{lemma}\label{lm:globalsharp}
Under the above notations, let $V=\bE\setminus \bF$. Pulling back one-forms defines a surjective bundle map
\[
q: U\times_V T^*V=T^*V|_U\rightarrow T^*U. 
\]
Then, over the open set $\mathcal U^\circ=\mathcal U\setminus D$, we have an equality of algebraic cycles
\[\Lambda^\sharp|_{\mathcal U^\circ}=(q^{-1}\Lambda)|_{\mathcal U^\circ}.\]
\end{lemma}
\begin{proof}
Recall that, locally on $\mathcal U^\circ$, the section $\bs: X\to \bE$ is given by $$\bs(x)=(x, f_1(x), \ldots, f_r(x)).$$ 
Thus, as the dual map of $q$, the pushfoward map on tangent bundle is of the form
\begin{equation}\label{eq_q1}
q^\vee: TU\to \bs^*TV=TV|_X, \quad v\mapsto v+\sum_{i=1}^r{df_i(v)}\bs^*(\partial_{t_i})
\end{equation}
for any tangent vector $v$ in $TU$. Locally on $\mathcal U^\circ$, under the above coordinate system, 
\[
T^*V=T^*U\times T^*\C^r=T^*U\times \C^r_{\underline{t}}\times \C^r_{\underline{s}}.
\]
Then as the dual map of \eqref{eq_q1}, the bundle map $q$ at a fiber $T_x^*V$ ($x\in U$) is of the form
\begin{equation}\label{eq_q2}
q: T^*V|_U\rightarrow T^*U, \quad \xi+\sum_{i=1}^r\lambda_idt_i\mapsto \xi+\sum_{i=1}^r\lambda_i df_i,
\end{equation}
where $\xi\in T_x^*U$. 

Now, we show that $\Lambda^\sharp|_{\mathcal U^\circ}\subset (q^{-1}\Lambda)|_{\mathcal U^\circ}$, or equivalently, $q(\Lambda^\sharp|_{\mathcal U^\circ})\subset \Lambda$. Since $\Lambda$ is conic, for any point $x\in \mathcal{U}^\circ$ and $v\in T^*_xU$ such that the pairing between $v$ and any element in $ \Lambda\cap T^*_xU$ is zero, by \eqref{eq_sharp} we need to show that 
\[
q\left(\xi+\sum_{i=1}^r s_id\log f_i+\sum_{i=1}^{r}s_i (-d\log t_i)\right)(v)=0
\]
for any $\xi\in\Lambda\cap T^*_xU$.
In fact, by \eqref{eq_q2}, we have
\begin{align*}
q\left(\xi+\sum_{i=1}^r s_id\log f_i+\sum_{i=1}^{r}s_i (-d\log t_i)\right)(v)=&\left(\xi+\sum_{i=1}^r s_id\log f_i\right)(v)-\sum_{i=1}^{r}s_i \frac{df_i}{t_i}(v)\\
=&\xi(v)+\sum_{i=1}^r s_i\left(\frac{df_i}{f_i}(v)-\frac{df_i}{t_i}(v)\right)\\
=&\xi(v)\\
=&0
\end{align*}
where the second last equality follows from the fact that the image of $\bs: U\to V$ is cut out by equations $t_i=f_i$ for $1\leq i\leq r$ and the last equality follows from the assumption that the pairing between $v$ and any element in $ \Lambda\cap T^*_xU$ is zero. 

Therefore, we have proved the inclusion that $\Lambda^\sharp|_{\mathcal U^\circ}\subset (q^{-1}\Lambda)|_{\mathcal U^\circ}$. Since $\Lambda$ is irreducible, the inclusion is between two irreducible closed $(n+r)$-cycles in $T^*\cU^\circ\times \C_{\underline{s}}^r$. Hence the inclusion must be an equality, that is, $\Lambda^\sharp|_{\mathcal U^\circ}= (q^{-1}\Lambda)|_{\mathcal U^\circ}$.
\end{proof}

The above lemma shows that the global cycle $\Lambda^\sharp\subset T^*(\bE,\bF)|_U$ is well-defined, and 
\begin{equation}
\Lambda^\sharp=q^{-1}\Lambda.
\end{equation} 

\begin{lemma}\label{lm:adclem}
If $D=\sum_{i=1}^r D_i$ is a simple normal crossing divisor on $X$, then the pullback map of logarithmic forms, $p$, is a surjective map of vector bundles over $X$:
\[
\begin{tikzcd}
\bs^*T^*(\bE,\bF)= T^*(\bE,\bF)|_X\arrow[rr,"p"]\arrow[dr]& & T^*(X,D)\arrow[dl] \\
& X &
\end{tikzcd}
\]
In particular, as restriction of $p$, the pullback map $q: T^*V|_U\rightarrow T^*U$ is also a surjective map of vector bundles. 
\end{lemma}

\begin{proof}
This follows from the fact that $\bs$ induces an embedding of log pairs $\bs: (X, D)\to (\bE, \bF)$. In other words, the image of $\bs: X\to \bE$ intersects the divisor $\bF$ transversally in the stratified sense. 

More precisely, we will show that there are (analytic) local coordinates of $\bE$, such that $\bF$ is defined by the product of a subset of the coordinates and the image of $\bs: X\to \bE$ is cut out by a disjoint subset of coordinates. Then the desired statement follows immediately. 

Choose a small analytic neighborhood $\mathcal{U}$ of a given point $y\in X$, such that  there exist holomorphic functions $f_1, \ldots, f_r$ on $\mathcal{U}$ defining the divisors $D_1, \ldots, D_r$. 
Let $x=(x_1, \ldots, x_n)$ be a set of coordinates of $\mathcal{U}$. Without loss of generality, we assume that $y\in D_1\cap \cdots \cap D_{r_0}$ and $y\notin D_i$ for $i>r_0$. Since $D$ is simple normal crossing, we can assume that $f_i=x_i$ for $1\leq i\leq r_0$. 
Then, locally, the embedding $\bs|_{\mathcal U}: \mathcal{U}\to \bE|_{\mathcal U}=\mathcal{U}\times \C^r_{\underline{t}}$ is given by 
\[
x\mapsto (x, f_1(x), \ldots, f_r(x))=(x, x_1, \ldots, x_{r_0}, f(x_{r_0+1}), \ldots, f(x_r)).
\]
The variables $(x, \underline{t})=(x_1, \ldots, x_n, t_1, \ldots, t_r)$ is a set of coordinates of $\bE|_{\mathcal U}=\mathcal{U}\times \C^r_{\underline{t}}$. Thus, 
\begin{equation}\label{eq_coordinate}
\big(t_1, \ldots, t_{r_0}, x_{r_0+1}, \ldots, x_r, t_1-x_1, \ldots, t_{r_0}-x_{r_0}, t_{r_0+1}-f_{r_0+1}, \ldots, t_r-f_r\big)
\end{equation}
also form a set of coordinates of $\bE|_{\mathcal U}=\mathcal{U}\times \C^r_{\underline{t}}$. Since $\mathcal{U}$ is centered at $y$ with $y\notin D_i$ for $i>r_0$, we can assume that $\mathcal{U}$ does not intersect $D_i$ for $i>r_0$. Now, the image of $\bs|_{\mathcal{U}}$ is cut out by
\[
t_1-x_1=0, \cdots, t_{r_0}-x_{r_0}=0, t_{r_0+1}-f_{r_0+1}=0, \cdots, t_r-f_r=0,
\]
and the divisor $\bF$ in $\bE|_\mathcal{U}$ is defined by $t_1\cdots t_{r_0}=0$. Therefore, the set of coordinates \eqref{eq_coordinate} satisfy the desired property. 
\end{proof}

\begin{remark}
Suppose $D$ is simple normal crossing but not smooth. If we take $\bE$ to be the line bundle 
$\mathcal{O}_X$ on $X$, then we can similarly define an irreducible divisor $\bF$ of $\bE$, the graph embedding $\bs: X\to \bE$, and the projection 
\[
p: T^*(\bE, \bF)|_X\to T^*(X, D).
\]
However, in this case, the map $p$ will not be surjective. For example, assume $\dim X=2$ and $(x_1, x_2)$ are locally coordinates of $X$ such that locally $D$ is defined by $x_1x_2=0$. One can easily see that the image of $p$ on the fiber $T^*(X, D)|_{(0, 0)}$ is 1-dimensional and spanned by $d\log(x_1x_2)$. 
Therefore, it is necessary to introduce a variable $t_i$ to each irreducible component $D_i$. 
\end{remark}

Similar to the construction in \cite[Section 3]{Gin}, we define $\overline{\Lambda^\sharp}$ to be the closure of $\Lambda^\sharp$ inside $\bs^*T^*(\bE,\bF)=T^*(\bE, \bF)|_{X}$. 
\begin{corollary}\label{cor_psharp}
Let $\Lambda\subseteq T^*U$ be a conic Lagrangian cycle.
If $D=\sum_{i=1}^r D_i$ is simple normal crossing, then 
\[\overline{\Lambda^\sharp}=p^{-1}(\overline\Lambda^{\log}),\]
where $\overline\Lambda^{\log}$ is the closure of $\Lambda$ inside $T^*(X,D)$.
\end{corollary}
\begin{proof}
By Lemma \ref{lm:adclem}, we have the following cartesian square
\[
\begin{tikzcd}
T^*V|_U \ar[r, "q"]\ar[d, hook]& T^*U \ar[d, hook]\\
 T^*(\bE, \bF)|_X \ar[r, "p"]&T^*(X, D),
\end{tikzcd}
\]
where $V=\bE\setminus \bF$ and both horizontal maps are surjective maps of vector bundles. Since $p^{-1}(\overline\Lambda^{\log})$ is closed and it contains $\Lambda^\sharp$, the inclusion $\overline{\Lambda^\sharp}\subset p^{-1}(\overline\Lambda^{\log})$ follows. On the other hand, both $\overline{\Lambda^\sharp}$ and $p^{-1}(\overline\Lambda^{\log})$ are irreducible $(n+r)$-cycles in $ T^*(\bE, \bF)|_X$. Hence, they must be equal to each other, that is, $\overline{\Lambda^\sharp}=p^{-1}(\overline\Lambda^{\log})$.
\end{proof}

For the rest of this section, we assume that the divisor $D=\sum_{i=1}^r D_i$ is simple normal crossing. 
Consider the following diagram of vector bundles on $X$,
\be\label{diag:factorpsi}
\begin{tikzcd}
T^*X\arrow[r,"\pi^*"] \arrow[rd, "\iota"]&T^*\bE|_X\arrow[r,"p'"]\arrow[d, "\phi_{\bE}"]& T^*X\arrow[d, "\phi"]\\
& T^*(\bE, \bF)|_X\arrow[r,"p"]&T^*(X, D),
\end{tikzcd}
\ee
where both $\phi$ and $\phi_\bE$ are defined by considering a one-form as a logarithmic one form, the map $p'$ is the pullback map of one-forms by $\bs$, the map $\pi^*$ is the pullback map of one-forms by $\pi: \bE\to X$, and $\iota\coloneqq \phi_\bE\circ \pi^*$. 

Since our local description of $\Lambda^\sharp$ in \eqref{eq_sharp} is the same as Ginsburg's $\Lambda^\sharp_{\underline{s}}$, we have the following theorem.
\begin{theorem}[{\cite[Theorem 3.2]{Gin}, see also \cite[Theorems 3.4]{FK}}]\label{thm_GFK}
Let $\sF^\centerdot$ be any constructible complex on $U$, with characteristic cycle   $CC(\sF^\centerdot)=\sum_{k}n_k\Lambda_k$. Then, as cycles, 
\[
CC(Rj_*(\sF^\centerdot))=\sum_{k}n_k\iota^*\overline{\Lambda^\sharp_k}.
\]
\end{theorem}
\begin{proof}
Given an irreducible conic Lagrangian cycle $\Lambda\subset T^*U$, by \eqref{eq:trivizlogpb}, 
\[
T^*(\bE, \bF)|_{\mathcal U}\cong T^*\mathcal{U}\times \C^r_{\underline{s}}.
\]
Under this trivialization, our definition of $\Lambda^\sharp$ is the same as the total space of $\Lambda^\sharp_s$ in \cite[Equation (3.3)]{FK}. Moreover, $\iota$ is a map of vector bundles over $X$, and under the above trivialization
\[
\iota: T_x^*X\to T_x^*(\bE, \bF), \quad \xi\to (\xi, 0, \ldots, 0),
\]
for any $x\in \mathcal U$. Thus, the image of $\iota$ is exactly cut out by equations $s_1=\cdots =s_r=0$. Therefore, our $\iota^*\overline{\Lambda^\sharp}$ is exactly equal to $\lim_{{\underline{s}}\to (0, \ldots, 0)}\Lambda_{\underline{s}}^\sharp$, as defined in \cite{FK}. Therefore, the desired equality follows from \cite[Theorem~3.4]{FK}.
\end{proof}

We are ready to prove the first main theorem \ref{thm_main}. 
\begin{proof}[Proof of Theorem \ref{thm_main}]
Let $\sF^\centerdot$ be a constructible complex on $U$, and assume that its characteristic cycle is of the form $CC(\sF^\centerdot)=\sum_k n_k \Lambda_k$. By Corollary \ref{cor_psharp} and Theorem~\ref{thm_GFK}, 
\[
CC(Rj_*(\sF^\centerdot))=\sum_{k}n_k\iota^* p^*\left(\overline{\Lambda_{k}}^{\,\log}\right).
\]
By the commutative diagram \eqref{diag:factorpsi}, since $p'\circ \pi^*=\id$, we have $\iota^* p^*=\phi^*$ and the above equation becomes
\[
CC(Rj_*(\sF^\centerdot))=\sum_{k}n_k\phi^*\left(\overline{\Lambda_{k}}^{\,\log}\right).
\]
By Lemma \ref{lemma_functoriality}
we have:
\[
c_*^{T^*X}\left(\phi^*\left(\overline{\Lambda_{k}}^{\,\log}\right)\right)=c_*^{T^*(X, D)}\left(\overline{\Lambda_{k}}^{\,\log}\right).
\]
Therefore, 
\[
c^{T^*X}_*\left( CC(Rj_*(\sF^\centerdot))\right)
=c^{T^*(X, D)}_*\left(\sum_k n_k\overline{\Lambda_{k}}^{\,\log}\right)=c^{T^*(X, D)}_*\Big(\overline{CC(\sF^\centerdot)}^{\log}\Big),
\]
and we have finished the proof. 
\end{proof}

%%%%%%%%%%%%%%%%%%%%%%%%%%%%

\section{Applications and Examples}\label{sec:4}
In this section, we make use of Theorem \ref{thm_main} for proving Theorems \ref{thm_CSM}, \ref{thm:ML-involution} and \ref{thm_MLb}. 

First, we apply Theorem~\ref{thm_main} to provide a dictionary between the class of the closure of the total space of critical points $\mathfrak{X}(Z)\subset \bP^n
\times \bP^n$ and  the Chern-Mather class $c_{Ma}(Z)$.

\begin{proof}[Proof of Theorem \ref{thm_CSM}]
To apply Theorem \ref{thm_main}, we let $X=\bP^n$ with homogeneous coordinates $[p_1, \ldots, p_n,p_+]$. Let $D_i=\{p_i=0\}$ for $0\leq i\leq n$, $D=\sum_{i=0}^n D_i$ and $U=X\setminus D=(\bC^*)^n$. 

First of all, we notice that the logarithmic cotangent bundle $E\coloneqq T^*(\bP^n, D)$ is a trivial rank $n$ vector bundle. Thus, we can identify the compactification $\overline{E}$ with $\bP^n\times \bP^n$, with the first factor being the base and the second being the fiber. Given a very affine variety $Z\subset U=(\bC^*)^n$, we have by definition that 
\begin{equation}\label{eq_reg}
\mathfrak{X}^\circ(Z)=T^*_{Z_{\textrm{reg}}}(\bC^*)^n.
\end{equation}
Let $\Lambda=T^*_Z(\bC^*)^n$ be the closure of $T^*_{Z_{\textrm{reg}}}(\bC^*)^n$ in $T^*(\bC^*)^n$. Then $\Lambda$ is a conic Lagrangian cycle in $T^*(\bC^*)^n$. 
Taking the closure of \eqref{eq_reg} in $E=T^*(\bP^n, D)$, we have an equality of algebraic cycles: 
\[
\mathfrak{X}(Z) \cap E =\overline{\Lambda}^{\log}.
\]
The closure of $\mathfrak{X}(Z) \cap E$ in $\overline{E}$ is exactly $\mathfrak{X}(Z)$. Hence, by \eqref{sh}, if 
\[
c_*^{T^*(\bP^n,D)}(\overline{\Lambda}^{\log})=\sum_{i=0}^d v_i [\bP^i] \in A_*(\bP^n),
\]
with $d=\dim Z$, then
\[
[\mathfrak{X}(Z)]=\sum_{i=0}^d v_i [\bP^i \times \bP^{n-i}] \in A_*(\bP^n\times \bP^n).
\]
Next consider $Eu_Z$ as a constructible function on $\bP^n$, with value equal to zero {outside $Z$}. In view of \eqref{equ}, this corresponds to the pushforward of $\Lambda$ under the open inclusion 
$j:(\bC^*)^n \hookrightarrow \bP^n$. Let $\phi:T^*\bP^n \to T^*(\bP^n,D)$ be the natural bundle map appearing in diagram \eqref{diag:factorpsi}. Then we have by Theorem \ref{thm_main}, Proposition \ref{pmc}, Lemma \ref{lemma_functoriality} and \eqref{cc} the following equalities in $A_*(\bP^n)$:
\begin{align*}
c_*^{T^*(\bP^n,D)}(\overline{\Lambda}^{\log})
=& c_*^{T^*\bP^n}(\phi^* \overline{\Lambda}^{\log}) \\
=& (-1)^d \cdot c_*^{T^*\bP^n}(CC(Eu_Z)) \\
=& (-1)^d \cdot \check{c}_{Ma}(Z).
\end{align*}
This is equivalent to saying that
\[
c_{Ma}(Z)=\sum_{i=0}^d (-1)^{d-i}v_i [\bP^i] \in A_*(\bP^n),
\]
which finishes the proof of Theorem \ref{thm_CSM}. 
\end{proof}

The next result, Theorem~\ref{thm_MLb}, is a stepping stone towards proving the Huh-Sturmfels conjecture of Theorem~\ref{thm:ML-involution}.
\begin{proof}[Proof of Theorem \ref{thm_MLb}]
First, we consider the compactification $(\C^*)^n\subset \bP^n$, with homogeneous coordinates $[p_1, \ldots, p_n, p_+]$ such that $x_i=\frac{p_i}{p_+}$. Then the hyperplane $\sH\subset \bP^n$ is defined by $p_+-p_1-\cdots -p_n=0$, or equivalently, $p_0=0$ with $p_+=p_0+\cdots+p_n$. 

Let $U=(\bC^*)^n\setminus \mathcal{H}$, and let $X=\bP^n$. Then the boundary divisor $D=X\setminus U$ is equal to the union of all coordinate hyperplanes and $\mathcal{H}$, which is a simple normal crossing divisor. Notice that for any ${{u}}\in \C^{n+1}$, the holomorphic 1-form 
\[
d\log l_{{u}}=u_1d\log x_1+\cdots +u_nd\log x_n+ u_0 d\log(1-x_1-\cdots -x_n)
\]
on $U$ extends to a logarithmic $1$-form on $X$, that is, a global section of $\Omega^1_X(\log D)$. Moreover, since the mixed Hodge structure on $H^1(U, \bQ)$ is of $(1, 1)$-type, there is a natural isomorphism 
\[
H^1(U, \C)\cong H^0(X, \Omega^1_X(\log D)),
\]
and hence $\dim H^0(X, \Omega^1_X(\log D))=n+1$. On the other hand, the $1$-forms $d\log l_{{u}}$ form a vector space of dimension $n+1$. Thus, there is a one-to-one correspondence between the $1$-forms $d\log l_{{u}}$ and the global sections of $\Omega^1_X(\log D)$. Under this correspondence, $H^0(X, \Omega^1_X(\log D))$ has a natural basis, given by 
\[
d\log x_1, \, \ldots,\,  d\log x_n,
\,\, \text{and}\,\, 
d\log(1-x_1-\cdots -x_n).
\]

Since the vector bundle $\Omega^1_X(\log D)$ is globally generated, the evaluation map 
\[
H^0(X, \Omega^1_X(\log D))\otimes_{\C} \cO_X\to \Omega^1_X(\log D)
\]
is surjective. Identifying $H^0(X, \Omega^1_X(\log D))$ with $\C^{n+1}$ using the above basis, we denote the corresponding map on the total spaces by 
\[
\rho: X\times \C^{n+1}\to T^*(X, D).
\]

The likelihood correspondence $\sL_Y$ is defined as the closure in $\bP^n \times \bP^{n+1}$ of
\[
\sL_Y^\circ\coloneqq\left\{({p}, {u})\in Y^\circ_{\textrm{reg}} \times \C^{n+1}\mid {p} \text{ is a critical point of } l_{{u}}|_{Y^\circ_{\textrm{reg}}}\right\},
\]
where $Y^\circ=Y\setminus \sH$. Thus, by definition, we have
\[
\sL_Y^\circ=\rho^{-1}\left(T^*_{Y^\circ_{\textrm{reg}}}U\right).
\]
Taking closure in $X\times \bC^{n+1}=\bP^{n}\times \bC^{n+1}$, we have
\begin{equation}\label{eq_LY}
\sL_Y\cap (\bP^{n}\times \bC^{n+1})=\rho^{-1}\left(\overline{T^*_{Y^\circ} U}^{\log}\right).
\end{equation}
Since $\bP^n\times \C^{n+1}$ is a trivial vector bundle over $\bP^n$, its fiber-wise compactification is equal to $\bP^n\times \bP^{n+1}$. Since $\sL_Y\cap (\bP^n \times \bC^{n+1})$ is a conic subvariety of $\bP^n\times \C^{n+1}$ and its closure in $\bP^n\times \bP^{n+1}$ is equal to $\sL_Y$, we get by \eqref{sh} that, if
\[
c_*^{\bP^n\times \C^{n+1}}\big(\sL_Y\cap (\bP^{n}\times \bC^{n+1})\big)=\sum_{i=0}^{\dim Y}\alpha_i[\bP^i] \in A_*(\bP^n),
\]
then
\[
[\sL_Y]=\sum_{i=0}^{\dim Y}\alpha_i[\bP^i\times \bP^{n+1-i}]\in A_*(\bP^n\times \bP^{n+1}). 
\]
Now, considering $Eu_{Y^\circ}$ as a constructible function on $X=\bP^n$ with value zero outside $Y^\circ$ and denoting its {signed Chern-Schwartz-MacPherson class by $\check{c}_{Ma}(Y^\circ) \in A_*(\bP^n)$,} we have
\begin{align*}
    c_*^{\bP^n\times \C^{n+1}}\big(\sL_Y\cap (\bP^{n}\times \bC^{n+1})\big)=&c_*^{T^*(X, D)}\left(\overline{T^*_{Y^\circ} U}^{\log}\right)\\
    =&(-1)^{\dim Y} \cdot \check{c}_{Ma}(Y^\circ),
\end{align*}
where the first equality follows from \eqref{eq_LY} and Lemma \ref{lemma_functoriality}, and the second equality follows from Theorem \ref{thm_main}, Proposition \ref{pmc} and \eqref{cc}. Finally, combining the above three displayed equations, we can conclude Theorem \ref{thm_MLb}. 
\end{proof}

We are now ready to prove the Huh-Sturmfels conjecture.
\begin{proof}[Proof of Theorem \ref{thm:ML-involution}]
Let $\varphi=Eu_{Y^\circ}$, considered as a constructible function on $\bP^n$. By Theorem \ref{thm_MLb} and Corollary \ref{cor_RW}, respectively, we have
\[
B_Y(\pp, \uu)=(-1)^{\dim Y}\gamma_{\varphi}\left(-\frac{\uu}{\pp}\right)\pp^n,\quad\text{and}\quad
S_Y(\pp, \uu)=(-1)^{\dim Y}\chi_{\varphi}\left(\frac{\uu}{\pp}\right)\pp^n,
\]
with $\gamma_\varphi$ and $\chi_\varphi$ as defined in Section \ref{inv}. 
By Corollary \ref{cor_Aluffi}, we have
\begin{align*}
    S_Y(\pp, \uu)=&(-1)^{\dim Y}\chi_{\varphi}\left(\frac{\uu}{\pp}\right)\pp^n\\
    =&(-1)^{\dim Y}\frac{\frac{\uu}{\pp}\gamma_{\varphi}(-\frac{\uu}{\pp}-1)+\gamma_{\varphi}(0)}{\frac{\uu}{\pp}+1}\pp^n\\
    =&(-1)^{\dim Y}\frac{{\uu}\gamma_{\varphi}(-\frac{\uu+\pp}{\pp})+\pp\gamma_{\varphi}(0)}{{\uu}+\pp}\pp^n\\
    =&\frac{\uu}{\uu+\pp}B_Y(\pp, \uu+\pp)+\frac{\pp}{\uu+\pp}B_Y(\pp, 0),
\end{align*}
which is one of the involution formulas. The other involution formula can be checked by similar computations. \end{proof}

Now we give two examples of the Huh-Sturmfels involution.
\begin{ex}[Singular cubic $3$-folds]
For a focused family of examples, we choose the following singular cubic  $3$-folds in 
$(\bC^*)^5$. 
We assume $X$ is defined by $f_1=x_1+\cdots+x_5-1$ and  a cubic polynomial $f_2$ of the form
\[
f_2= (x_2+x_3+x_4+x_5)^3-L\cdot(x_1+x_2+x_3+x_4)^2
\]
where $L$ is some linear polynomial. The singular locus of $X$ is independent of $L$; it is 
\[
V\left(x_1+\cdots +x_5-1, \, 
x_2 + x_3 + x_4 + x_5, \,
x_1 - x_3 - x_4 - x_5 \right)\cap (\mathbb{C}^*)^5.
\]

\begin{table}[htb!]\label{table:cubic-three-fold}
    \centering
    \begin{tabular}{c|Hc}
%    $L$& sing. locus $X$ & $B_{Y_k}$ and $S_{Y_k}$ \\
    $L$& sing. locus $X$ & $B_{Y_k}$ and $S_{Y_k}$ \\
    \hline

    $2\,x_{1}+3\,x_{2}+5\,x_{3}+7\,x_{4}$    
    &$\left\{\texttt{ideal}{}\left(p{3}-\textit{ps},\,p{1}+p{2}+p{4}+\textit{ps},\,p{0}-\textit{ps}\right)\right\}$
    &$19\,\pp^{4}+15\,\pp^{3}\uu+9\,\pp^{2}\uu^{2}+3\,\pp\,\uu^{3}$\\
    &&$19\,\pp^{4}+27\,\pp^{3}\uu+15\,\pp^{2}\uu^{2}+3\,\pp\,\uu^{3}$\\
    \hline

    $x_{0}+x_{1}$    
    &$\left\{\texttt{ideal}{}\left(p{3}-\textit{ps},\,p{1}+p{2}+p{4}+\textit{ps},\,p{0}-\textit{ps}\right)\right\}$
    &$11\,\pp^{4}+12\,\pp^{3}\uu+9\,\pp^{2}\uu^{2}+3\,\pp\,\uu^{3}$\\
    &&$11\,\pp^{4}+24\,\pp^{3}\uu+15\,\pp^{2}\uu^{2}+3\,\pp\,\uu^{3}$\\
    \hline

    $x_{0}$    
    &$\left\{\texttt{ideal}{}\left(p{3}-\textit{ps},\,p{1}+p{2}+p{4}+\textit{ps},\,p{0}-\textit{ps}\right)\right\}$
    &$6\,\pp^{4}+6\,\pp^{3}\uu+6\,\pp^{2}\uu^{2}+3\,\pp\,\uu^{3}$\\
    &&$6\,\pp^{4}+15\,\pp^{3}\uu+12\,\pp^{2}\uu^{2}+3\,\pp\,\uu^{3}$\\
    \hline

    $x_{1}$    
    &$\left\{\texttt{ideal}{}\left(p{3}-\textit{ps},\,p{1}+p{2}+p{4}+\textit{ps},\,p{0}-\textit{ps}\right)\right\}$
    &$6\,\pp^{4}+8\,\pp^{3}\uu+7\,\pp^{2}\uu^{2}+3\,\pp\,\uu^{3}$\\
    &&$6\,\pp^{4}+18\,\pp^{3}\uu+13\,\pp^{2}\uu^{2}+3\,\pp\,\uu^{3}$\\
    \hline

    $x_{4}$    
    &$\left\{\texttt{ideal}{}\left(p{3}-\textit{ps},\,p{1}+p{2}+p{4}+\textit{ps},\,p{0}-\textit{ps}\right)\right\}$
    &$5\,\pp^{4}+7\,\pp^{3}\uu+7\,\pp^{2}\uu^{2}+3\,\pp\,\uu^{3}$\\
    &&$5\,\pp^{4}+17\,\pp^{3}\uu+13\,\pp^{2}\uu^{2}+3\,\pp\,\uu^{3}$\\
    \hline

    $x_1+x_2+x_3+x_4+x_5$    
    &$\left\{\texttt{ideal}{}\left(p{3}-\textit{ps},\,p{1}+p{2}+p{4}+\textit{ps},\,p{0}-\textit{ps}\right)\right\}$
    &$5\,\pp^{4}+5\,\pp^{3}\uu+5\,\pp^{2}\uu^{2}+3\,\pp\,\uu^{3}$\\
    &&$5\,\pp^{4}+13\,\pp^{3}\uu+11\,\pp^{2}\uu^{2}+3\,\pp\,\uu^{3}$\\
    \hline
    
    \end{tabular}
    \caption{ML Bidegrees and sectional degrees of singular cubics}
    \label{tab:my_label}
\end{table}

\end{ex}

\begin{ex}[Independence models]
The next example is motivated by statistics.
Consider the variety of order $k$ 
rank one $2\times 2\times\cdots \times 2$ tensors, which we  denote by $Y_k$. In algebraic statistics this variety is known as an independence model, and is known to have ML degree one.

We use numerical algebraic geometry to determine the ML bidegrees of $Y_k$. The results are recorded in Table~\ref{table:rank-one-tensors}.
Specifically, we use a parameterization and a numerical implementation of \cite[Algorithm 18]{HKS2005}.
By Theorem~\ref{thm:ML-involution}, we also find the sectional ML degrees.
In this example, computing the sectional ML degrees using numerical computation is more difficult because we don't have a nice parameterization of $Y_k\cap L$ where $L$ is a generic linear space.
In general, our computations and the \href{https://oeis.org/search?q=1\%2C4\%2C15\%2C64&language=english&go=Search}{OEIS} suggest these formulas:
\[
B_{Y_k}(\pp,\uu) = \pp^{n} + 
\sum _{i=1}^k \left(
i! \, \binom{k}{i}   \pp^{n-i}\uu^i 
\right)
\quad \text{ and }\quad
S_{Y_k}(\pp,\uu) = \pp^{n} + 
\sum _{i=1}^k i! \, \binom{k}{i}  \cdot \pp^{n-1}\uu+\cdots.
\]
\begin{table}[htb!]\label{table:rank-one-tensors}
    \centering
    \begin{tabular}{c|c|c}
    $k$ =$\dim(Y_k)$ &$n=2^k-1$ & $B_{Y_k}$ and $S_{Y_k}$ \\
    \hline

    $2$ & $4$ &
        $\pp^{3}+2\,\pp^{2}\uu+2\,\pp\,\uu^{2}$\\
    &&  $\pp^{3}+4\,\pp^{2}\uu+2\,\pp\,\uu^{2}$\\
    \hline
    
    $3$  & $8$ & 
         $\pp^7+3\pp^6\uu+6\pp^5\uu^2+6\pp^4\uu^3$\\
    && $\pp^7+15\pp^6\uu+18\pp^5\uu^2+6\pp^4\uu^3$\\
    \hline

    $4$ & $16$ &  
      $\pp^{15}+4\pp^{14}\uu+12\pp^{13}\uu^2+24\pp^{12}\uu^3+24\pp^{11}\uu^{4}$\\
    &&$\pp^{15}+64\pp^{14}\uu+132\pp^{13}\uu^2+96\pp^{12}\uu^3+24\pp^{11}\uu^4$\\
    \hline
    \end{tabular}
    \caption{ML Bidegree for rank one tensors}
    \label{tab:my_label}
\end{table}
\end{ex}

%%%%%%%%%%%%%%%%%%%%%%%%%%%%%

\bibliographystyle{abbrv}

\end{document}